\documentclass[12pt,reqno]{amsart}
\usepackage{fullpage,amsfonts,amsmath,amssymb}
\usepackage{verbatim}
\usepackage{hyperref}
\usepackage{thmtools}
\usepackage{xcolor}
\newcommand{\vertiii}[1]{{\left\vert\kern-0.25ex\left\vert\kern-0.25ex\left\vert #1 
		\right\vert\kern-0.25ex\right\vert\kern-0.25ex\right\vert}}

\hypersetup{colorlinks=true,linkcolor=red,citecolor=cyan}

\newcommand{\R}{\mathbb R}
\newcommand{\C}{\mathbb C}
\newcommand{\Z}{\mathbb Z}%
\newcommand{\N}{\mathbb N}

\newcommand{\T}{\mathbb{T}}

\newtheorem{theorem}{Theorem}[section]
\newtheorem{lemma}[theorem]{Lemma}

\theoremstyle{definition}

\theoremstyle{definition}
\newtheorem{remark}[theorem]{Remark}

\numberwithin{equation}{section}

\begin{document}
\title{$L^p$-boundedness of pseudo-differential operators on homogeneous trees}
\author{Tapendu Rana and Sumit Kumar Rano}
\address{Tapendu Rana  \endgraf Department of Mathematics,	\endgraf Indian Institute of Technology Bombay, 	\endgraf Powai, Mumbai 400076, Maharashtra, India.} \email{tapendurana@gmail.com}

\address{Sumit Kumar Rano \endgraf Stat-Math Unit,	\endgraf Indian Statistical Institute,	\endgraf 203 B. T. Road, Kolkata 700108, India.} \email{sumitrano1992@gmail.com} 

	
\begin{abstract}
The aim of this article is to study the $L^{p}$-boundedness of pseudo-differential operators on a homogeneous tree $ \mathfrak{X} $. For $p\in (1,2)$, we establish a connection between the $L^{p}$-boundedness of the pseudo-differential operators on  $ \mathfrak{X} $ and that on the group of integers $\mathbb{Z}$.  We also prove an analogue of the Calderon-Vaillancourt theorem in the setting of homogeneous trees, for $p\in(1,\infty)\setminus\{2\}$.
\end{abstract}
	
\maketitle

\section{Introduction}
The systematic study of pseudo-differential operators has drawn lots of motivation from partial differential equations, quantum mechanics and signal analysis. Indeed, the pioneering works on this subject in the 1960s, as explored for example by H\"ormander \cite{Hormander_1965,Hormander_Ann_1966} and Kohn-Nirenberg \cite{Kohn_Nirenberg_1965}, were guided by a deep connection to elliptic and hypo elliptic equations.  The boundedness results on the classical spaces of harmonic analysis play a special role due to their implications in the regularity of the solutions for appropriate equations. Recently, the calculus of pseudo-differential operators on discrete spaces  has gained popularity, see for e.g. \cite{Botchway_Ruzhansky_2020,Masson_14} due to their natural connection with various problems of quantum ergodicity and in the discretization of continuous problems.

The aim of this article is to study the theory of pseudo-differential operators on homogeneous trees. A homogeneous tree $\mathfrak{X}$ of degree $q+1$ is a connected graph with no loops, in which every vertex is connected to $q+1$ other vertices. Let us first discuss the theory of pseudo-differential operators on homogeneous trees of degree two, that is the group of integers $\mathbb{Z}$. Given a bounded measurable function $\psi$ on $\mathbb{Z}\times\mathbb{T}$, one considers the pseudo-differential operator $T_{\psi}$ defined by
\begin{equation}\label{pdoz1}
T_{\psi}f(l)=\frac{1}{\tau}\int\limits_{\mathbb{T}}\psi(l,s)\mathcal{F}f(s)q^{ils}~ds,\quad\text{for all }l\in\mathbb{Z},
\end{equation}
for a finitely supported function $f$ on $\mathbb{Z}$, where $\tau=2\pi/\log q$, $\mathbb{T}=\mathbb{R}/\tau\mathbb{Z}$ and $\mathcal{F}f$ denotes the Fourier transform of $f$ defined by
\begin{equation}\label{eqn_defn_Foruier_transform}
	\mathcal{F}f(s)=\sum\limits_{d\in\mathbb{Z}}f(d)q^{-ids},\quad\text{for all }s\in\mathbb{T}.
\end{equation}
Such a function $\psi$ is said to be the symbol of the pseudo-differential operator $T_{\psi}$. Alternatively, using \eqref{eqn_defn_Foruier_transform} one can also write \eqref{pdoz1} as follows:
\begin{equation}
	 T_{\psi}f(l)= \sum_{d\in \Z}f(d) \kappa (l,l-d),
\end{equation}
where $$ \kappa(l,l-d) =\frac{1}{\tau} \int\limits_\T \psi (l,s) q^{is(l-d)}.$$
If $T_{\psi}$ is a bounded operator from $L^{p}(\mathbb{Z})$ to itself, we will denote its operator norm by $\vertiii{\psi(\cdot,\cdot)}_{p}$, defined as
$$\vertiii{\psi(\cdot,\cdot)}_{p}  = \sup_{\| f\|_{L^p(\Z)} =1} \|T_\psi f\|_{L^p(\Z)}.$$
When $\psi$ is independent of the space variable, we get the Fourier multiplier operator. We shall denote such a Fourier multiplier operator on $\mathbb{Z}$ by $T_{m}$, which is defined by
\begin{equation}\label{multiplierz}
T_{m}f(l)=\frac{1}{\tau}\int\limits_{\mathbb{T}}m(s)\mathcal{F}f(s)q^{ils}~ds.
\end{equation}
S. Molahajloo studied the pseudo-differential operators on $\mathbb{Z}$ (see \cite{MR2664583}), where the author among other interesting results established sufficient conditions on the Fourier transform of the symbol with respect to the variable on $\mathbb{T}$ to ensure the $L^{p}$-boundedness. For higher dimensional analogue of the above result and a comprehensive pseudo-differential calculus on $ \Z^n $, we refer to \cite{Botchway_Ruzhansky_2020,MR2831664} and the references therein. On the other hand, the authors in \cite{MR2104184} studied pseudo-differential operators on $\mathbb{Z}^{n}$ with symbols which are bounded on $\mathbb{Z}^{n}\times\mathbb{T}^{n}$ together with their derivatives with respect to the second variable. More precisely, they proved the following analogue  of the Calderon-Vaillancourt theorem on $\mathbb{Z}$ (see \cite[Theorem 2.8]{MR2104184}):
\begin{equation}\label{cvonz}
\vertiii{\psi(\cdot,\cdot)}_{p}\leq C\sup\limits_{(l,s)\in\mathbb{Z}\times\mathbb{T},~k=0,1,2}~\left|\frac{d^{k}}{ds^{k}}\psi(l,s)\right|,\quad\text{for all }1\leq p\leq\infty.
\end{equation}

We are concerned with the case $q\geq 2$. The reason for handling this case separately can be ascribed to the difference between the analysis on $\mathbb{Z}$, which has a polynomial growth and that of others which have an exponential volume growth. We begin by discussing the case of multipliers on $\mathfrak{X}$. For unexplained notations, we refer to Section \ref{section2}. Corresponding to a bounded measurable function $m$ on $\mathbb{T}$, the associated Fourier multiplier operator on $\mathfrak{X}$ is defined by
\begin{equation}\label{multiplierx}
T_{m}f(x)=\int\limits_{\mathbb{T}}\int\limits_{\Omega} m(s)\widetilde{f}(s,\omega)p^{1/2-is}(h,\omega)|\mathbf{c}(s)|^{-2}~d\nu(\omega)~ds.
\end{equation}
The $L^{p}$-boundedness of the multiplier operator on homogeneous trees has been studied by many authors, see for instance \cite{Meda_Stefano_2019,Cowling_Meda_98,Cowling_Meda_99}. A celebrated result of J.L. Clerc and E.M. Stein \cite{Clerc_Stein}, in the context of homogeneous trees states that if $T_{m}$ is a bounded operator on $L^{p}(\mathfrak{X})$, then $m$ necessarily extends to a bounded holomorphic function on the strip $S_{p}^{\circ}$, where
$$S_{p}^{\circ}=\{z\in\C: |\Im z|<\delta_p\}\quad\text{and}\quad \delta_p=\left|\frac{1}{p}-\frac{1}{2}\right|,\quad\text{for }p\in[1,\infty).$$
This necessary condition was sharpened by M. Cowling, S. Meda and A. G. Setti \cite[Theorem 2.1]{Cowling_Meda_99} who proved that if $T_{m}$ is a bounded operator on $L^{p}(\mathfrak{X})$, then the function $s\mapsto m(s+i\delta_{p})$ defines a bounded multiplier operator of the form (\ref{multiplierz}) on $L^{p}(\mathbb{Z})$. They also proved the converse by using some additional condition on the multiplier $m$. However, the most important result in this direction, which concerns us, is the following work of D. Celotto,  S. Meda and B. Wr\'{o}bel:

\begin{theorem}[{{\cite[Page 176]{Meda_Stefano_2019}}}]\label{multipliermainresult}
For $p\in[1,\infty)\setminus\{2\}$, the following are equivalent.
\begin{enumerate}
\item The operator $T_{m}$ defined by \eqref{multiplierx} is bounded on $L^{p}(\mathfrak{X})$.
\item The multiplier $m$ extends to a Weyl-invariant function on the strip $S_{p}^{\circ}$ and the function $s\mapsto m(s+i\delta_{p})$ defines a bounded multiplier operator of the form \eqref{multiplierz} on $L^{p}(\mathbb{Z})$.
\end{enumerate}
\end{theorem}

\noindent Here and subsequently, a function $\eta$ defined on $S^{\circ}_{p}$ is said to be Weyl-invariant if $\eta$ is holomorphic on $S^{\circ}_{p}$, $\eta(z)=\eta(-z)$ and $\eta(z)=\eta(z+\tau)$ for all $z$ in $S^{\circ}_{p}$.

Now that the characterization of  $ L^p $ Fourier multipliers is completely settled, it is natural to inquire, what happens if one replaces the multiplier $ m(s) $ by a symbol $ \Psi(x,s)? $ We find some natural conditions on the symbol of a pseudo-differential operator on $\mathfrak{X}$ which implies the continuity of the corresponding operator. Explicitly, we contribute in two ways:
\begin{enumerate}
\item  Firstly, we shall establish a connection between the $L^{p}$-boundedness of the pseudo-differential operators on $\mathfrak{X}$ and that on the group of integers $\mathbb{Z}$. 
\item  Secondly, we provide a sufficient condition on the symbol $\Psi$, which will guarantee the $L^p$-boundedness of the corresponding pseudo-differential operator on homogeneous trees.
\end{enumerate} Let $\Psi:\mathfrak{X}\times\mathbb{T}\rightarrow\mathbb{C}$ be a measurable function. We define the pseudo-differential operator associated with the symbol $\Psi$, by
\begin{equation}\label{pdotree}
	T_{\Psi}f(h)=c_{G}\int\limits_{\mathbb{T}}\int\limits_{\Omega}\Psi(h,s)\widetilde{f}(s,\omega)p^{1/2-is}(h,\omega)|\mathbf{c}(s)|^{-2}~d\nu(\omega)~ds.
\end{equation}
Then we have proved the following result for the case $1<p<2$. 
\begin{theorem}\label{mainresult1}
	Let $1<p<2$. Suppose that $\Psi:\mathfrak{X}\times S_{p}\rightarrow\mathbb{C}$ is a function satisfying the following properties:
\begin{enumerate}
\item For each $x$ in $\mathfrak{X}$, $z\mapsto\Psi(x,z)$ is Weyl-invariant on the strip $S^{\circ}_{p}$ and continuous on $S_{p}$.
\item The symbol $\Psi$ is uniformly bounded on $\mathfrak{X}\times S_{p}$ , that is,
$$\sup\limits_{x\in\mathfrak{X},\,z\in S_{p}}|\Psi(x,z)|<\infty.$$
\item For each $n$ in $N$, $(l,s) \mapsto \Psi(n\sigma^{l},s-i\delta_{p})$ defines a bounded pseudo-differential operator from $L^{p}(\mathbb{Z})$ to itself and
$$\sup\limits_{n\in N}\vertiii{\Psi_{-\delta_{p}}(n\cdot,\cdot)}_{p}<\infty.$$
\end{enumerate}
Then $T_{\Psi}$ defines a bounded operator from $L^{p}(\mathfrak{X})$ to itself. Moreover, there exists a constant $C_{p}>0$ such that
$$\|T_{\Psi}f\|_{L^{p}(\mathfrak{X})}\leq C_{p}\left(\|\Psi\|_{L^{\infty}(\mathfrak{X}\times S_{p})}+\sup\limits_{n\in N}\vertiii{\Psi_{-\delta_{p}}(n\cdot,\cdot)}_{p}\right)\|f\|_{L^{p}(\mathfrak{X})}.$$
\end{theorem}
We next turn to the case $2<p<\infty$. The $L^{p}$-boundedness of the multiplier operator \eqref{multiplierx} (for $p>2$) follows from a straightforward duality argument, using the fact that the kernel of the operator is radial. However, the duality argument does not work for pseudo-differential operators. We overcome this obstacle by imposing a rather stronger condition on the symbol $\Psi$. More precisely, we prove the following analogue of  Calderon-Vaillancourt theorem (\cite[Theorem 2.8]{MR2104184}) in the setting of homogeneous trees.
\begin{theorem}\label{mainresult2}
Let $2<p<\infty$. Suppose that $ \Psi:\mathfrak{X}\times S_{p}\rightarrow\mathbb{C}$ is a function satisfying the following properties:
	\begin{enumerate}
		\item For each $x$ in $\mathfrak{X}$, $z\mapsto\Psi(x,z)$ is Weyl-invariant on the strip $S^{\circ}_{p}$.
		\item The functions $d^{k}\Psi/dz^{k}$  extend continuously on $S_{p}$ and satisfy the differential inequalities
\begin{equation}\label{eqn_regular_condn_Psi_thm2}
			\left|\frac{d^{k}}{dz^{k}}\Psi(x,z)\right|\leq C_{k},\quad\text{for all }k=0,1,2;\text{ for all }x\in\mathfrak{X},z\in S_{p}.
\end{equation}
	\end{enumerate}
	Then $T_{\Psi}$ defines a bounded operator from $L^{p}(\mathfrak{X})$ to itself. Moreover, there exists a constant $C_{p}>0$ such that
$$\|T_{\Psi}f\|_{L^{p}(\mathfrak{X})}\leq C_{p}\left(~\sup\limits_{(x,z)\in\mathfrak{X}\times S_{p},~k=0,1,2}~\left|\frac{d^{k}}{dz^{k}}\Psi(x,z)\right|~\right) \|f\|_{L^{p}(\mathfrak{X})},$$
for all $f\in L^{p}(\mathfrak{X})$.
\end{theorem}

\begin{remark}
\begin{enumerate}
\item When $ \Psi$ is a multiplier, then one can write $ T_\Psi $ as a convolution operator with a $ K $-biinvariant kernel. This plays a crucial role in the multiplier case. However, in the case of pseudo-differential operators, one	cannot use such method due to an extra variable ``$x$" in the symbol $ \Psi(x,s) $. This is the fundamental difference between the multiplier case and our situation. 
\item In comparison to the results of the lattice $ \Z $, the holomorphic extension property of the symbol $\Psi$ is a new and necessary condition  for the pseudo-differential operator $T_{\Psi}$ to be bounded on $L^p(\mathfrak{X})$, as in the case of multipliers on $\mathfrak{X}$.
\item  A similar result is also known on rank one noncompact symmetric spaces \cite [Theorem 1.6]{PR_PDO_22}. However it is worth mentioning that for $ p\in (1,2) $, we get an improvement in the theorem above. This comes as a consequence of the discrete nature  of the homogeneous tree. 
\item In \cite{Masson_14}, Masson introduced pseudo-differential operators on $\mathfrak{X}$ associated with a more general symbol class and proved their $L^{2}$-boundedness. The key ingredient of the proof is that the kernels of the operators have a rapid decay property. The proof of the rapid decay property is  essentially a paraphrase of the $L^{2}$-Schwartz space isomorphism theorem given in \cite[Theorem 2]{MR1680799}. We remark that such an isomorphism theorem is not known for $p\neq 2$.
\item  We would like to mention that Theorem \ref{mainresult1} is sharper than Theorem \ref{mainresult2} in the sense that if $ \Psi $ satisfies the hypothesis of Theorem \ref{mainresult2} for $ 1<p<2 $, then using \eqref{cvonz}, $ T_\Psi $ can be extended as a $ L^p(\mathfrak{X})$ bounded operator.
\end{enumerate}
\end{remark}	

The article is organized as follows: Section \ref{section2} sets the necessary background of harmonic analysis on homogeneous trees. In Section  \ref{PDO_on_tree_1<p<2}, we shall express $T_{\Psi}$ as an integral operator and prove Theorem \ref{mainresult1}. Finally, in Section \ref{bddness_of_T_Psi_p>2} we will prove Theorem \ref{mainresult2}.

\section{Notation and preliminaries}\label{section2}

\subsection{Generalities}
The letters $\N$, $\Z$, $\R$ and $\C$ will respectively denote the set of all natural numbers, the ring of integers, and the fields of real and complex numbers. We denote the set of all non-negative integers by $ \Z_+ $. For $ z \in \C $, we use the notations $\Re z$ and $\Im z$ for real and imaginary parts of $z$ respectively. We shall follow the standard practice of using the letters $C$, $C_1$, $C_2$ etc. for positive constants, whose value may change from one line to another. Occasionally the constants will be suffixed to show their dependencies on important parameters. For any Lebesgue exponent $ p \in (1,\infty)$, let $p^{\prime}$ denote the conjugate exponent $p/(p-1)$. We further assume $p^{\prime}=\infty$ when $p=1$ and vice-versa. For $p\in[1,\infty)$, let
$$\delta_p=\left|\frac{1}{p}-\frac{1}{2}\right|,\quad\text{and}\quad S_{p}=\{z\in\mathbb{C}:|\Im z|\leq\delta_{p}\}.$$
We assume $\delta_{\infty}=1/2$. It is important to note that $\delta_{p}=\delta_{p^{\prime}}$ and $S_{p}=S_{p^{\prime}}$ for all $p\in[1,\infty]$. We shall henceforth write $S^{\circ}_{p}$ and $\partial S_{p}$ to denote the usual topological interior and the boundary of $S_{p}$ respectively. Given a  function $F$ on $S_{p}$, for every $v\in [-\delta_{p}, \delta_{p}]$, we denote by $F_{v}$, the function on $\R$ which is defined by $F_v(u) = F(u+iv)$.

\subsection{Homogeneous Trees}
Here we review some general facts about the homogeneous trees, most of which are already known. Details regarding harmonic analysis on trees can be found in the books \cite{Alessandro_Nebbia,Talamanca_Alessandro_1982}.

As described earlier, a homogeneous tree $\mathfrak{X}$ of degree $q+1$ is a connected and acyclic graph, in which every vertex has $q+1$ neighbours. We identify $\mathfrak{X}$ with the set of all vertices where the natural distance $d(x, y)$ between any two vertices $x$ and $y$ is defined as the number of edges between them. Let $o$ be a fixed but arbitrary reference point in $\mathfrak{X}$. We shall henceforth write $|x|$ to denote $d(o,x)$. The tree $\mathfrak{X}$ being discrete is naturally equipped with the counting measure. Let $G$ be the group of isometries of the metric space $(\mathfrak{X}, d)$ and $K$ be the stabilizer of $o$ in $G$. It is known that $K$ is a maximal compact subgroup of $G$. The map $ g\mapsto g\cdot o$ identifies $\mathfrak{X}$ with the coset space $G/K$, so that functions on $\mathfrak{X}$ correspond to $K$-right invariant functions on $G$. Furthermore, the radial functions on $\mathfrak{X}$, that is, functions which only depend on $|x|$, correspond to $K$-biinvariant functions on $G$.

\subsection{The height function and the boundary of \texorpdfstring{$\mathfrak{X}$}{X}}
An infinite geodesic ray $\omega$ is a one-sided sequence $\{\omega_{j}:j=0,1,\ldots\}$ where the $\omega_{j}$ are in $\mathfrak{X}$ and $d(\omega_i, \omega_j)=|i-j|$ for all non-negative integers $i$ and $j$. We say that an infinite geodesic $\omega$ starts from $x$ if $\omega_{0}=x$. For a given infinite geodesic ray $\omega$, we define the associated height function $ h_\omega :\mathfrak{X} \rightarrow\Z$  by
\begin{equation}\label{heightfunction}
h_\omega(x)=\lim\limits_{j \rightarrow \infty}\left(  j- d(x, \omega_j)\right).
\end{equation}
For details, we refer to \cite{Cowling_Meda_98}. The height function $h_{\omega}$ is the discrete analogue of the Busemann function in Riemannian geometry. We note that for each $x$ and $\omega$, the sequence in (\ref{heightfunction}) is eventually constant and hence the limit exists. Furthermore, for every $m \in \Z$, we define the $\omega$-horocycles
\begin{equation}\label{horocycle}
\mathfrak{H}(\omega,m) =\{ x \in \mathfrak{X} :  h_{\omega}(x)= m \},
\end{equation}
and see that $\mathfrak{X}$ decomposes into the disjoint union
\begin{equation}\label{horocycledecomposition}
\mathfrak{X}=\bigcup\limits_{m\in\mathbb{Z}}\mathfrak{H}(\omega,m) .
\end{equation}

Two infinite geodesics $\omega=\{\omega_j:j=0,1,\ldots\}$ and $\omega^{\prime}=\{\omega^{\prime}_j:j=0,1,\ldots\}$ are said to be equivalent if they meet at infinity, that is, there exists  $k$, $N \in \N $ such that for all $ j\geq N$, $\omega_j =\omega^{\prime}_{j+k}$. This identification is an equivalence relation and partitions the set of all infinite geodesics into equivalence classes. The equivalence class of a generic geodesic ray $\omega$ will henceforth be denoted by $[\omega]$. In every equivalence class, there exists a unique geodesic ray starting from $o$. The boundary $\Omega$ is defined as the set of all infinite geodesic rays starting at $o$. It is known that $K$ acts transitively on $\Omega$ via the map $k\mapsto k\cdot\omega$.

\subsection{The Groups \texorpdfstring{$A$}{A} and \texorpdfstring{$N$}{N}}
From now on, we fix a doubly infinite geodesic $\omega_{0}$ of the form $\{\omega^{0}_{j}:j\in\mathbb{Z}\}$ such that $\omega^{0}_{0}=o$ and $d(\omega^{0}_i, \omega^{0}_j)=|i-j|$ for all integers $i$ and $j$. Furthermore, we will denote the infinite geodesic $\omega^{+}_{0}$ by $\{\omega^{0}_{j}:j=0,1,\ldots\}$. Let $\sigma$ in $G$ be a translation of step one along the doubly infinite chain $\omega_{0}$. More explicitly, let $\sigma(\omega^{0}_{j})=\omega^{0}_{j+1}$ for all $j$ in $\mathbb{Z}$. We denote by $A$, the subgroup generated by $\sigma$ and define it as
$$A=\{\sigma^{j}:j\in\mathbb{Z}\}.$$
It is important to note that $A$ is abelian and isomorphic to the group $\mathbb{Z}$. Let $G_{\omega^{+}_{0}}$ be the stabilizer of the equivalence class $[\omega^{+}_{0}]$ in $G$. We define the group $N$ by
\begin{equation}\label{groupn}
N=\{n\in G^{+}_{\omega_{0}}:n\cdot x=x\text{ for some }x\in\mathfrak{X}\}.
\end{equation}
It is well known that the group $N$ is locally compact. We shall denote the Haar measure on $N$ by $\mu$. From \cite[Lemma 3.3]{Alessandro_2002} it is evident that $N$ is unimodular and the measure $\mu$ is normalised by the condition $\mu(N\cap K)=1$. Moreover, it was proved in \cite[Corollary 3.2]{Alessandro_2002} that the subgroup $N$ acts transitively on every $\omega^{+}_{0}$-horocycle $\mathfrak{H}(\omega^{+}_{0},m)$ defined by (\ref{horocycle}). By using this fact together with the horocyclic decomposition (\ref{horocycledecomposition}), Veca \cite[Theorem 3.5]{Alessandro_2002} proved the following Iwasawa-type decomposition of the group $G$ and arrived at the following integral formula for functions on $G$.

\begin{theorem}\label{iwasawa}
Let $G$, $N$, $K$ and $A$ be as defined above. Then for every $g\in G$, there exist $n\in N$, $j\in\mathbb{Z}$ and $k\in K$ such that $g=n\sigma^{j}k$. Furthermore, if $f$ is a compactly supported function defined on $G$, then
	$$\int\limits_{G}f(g)~dg=\int\limits_{N}\sum\limits_{j\in\mathbb{Z}}~\int\limits_{K}f(n\sigma^{j}k)~q^{-j}~dk~d\mu(n).$$
\end{theorem}

It is also known that the subgroup $A$ acts on $N$ by conjugation (see \cite[Lemma 3.8]{Alessandro_2002}). Moreover, for a compactly supported function $f$ on $N$, we have
\begin{equation}\label{modularfunction}
\int\limits_{N}f(\sigma^{j}n\sigma^{-j})~d\mu(n)=q^{-j}\int\limits_{N}f(n)~d\mu(n)
\end{equation}
We endow the set $NA$ with the binary operation induced by this conjugation and call it as the semi-direct product $NA$. By Theorem \ref{iwasawa}, we may identify $K$-right invariant functions on $G$ with functions on $NA$. In fact, the corresponding $L^{p}$-norms coincide. For the sake of simplicity, we denote the Haar measure $d\mu(n)$ by $dn$. 

\subsection{The spherical function and Fourier transform on \texorpdfstring{$\mathfrak{X}$}{X}}
 On the boundary $\Omega$, there exists a unique $K$-invariant, $G$-quasi-invariant probability measure $\nu$ and the Poisson kernel $p(g\cdot o,\omega)$ is defined to be the Radon-Nikodym derivative $d\nu(g^{-1}\cdot\omega)/d\nu(\omega)$. The Poisson kernel can be explicitly written as
$$p(x,\omega)=q^{h_{\omega}(x)},\quad\text{for all }x\in\mathfrak{X}\text{ and for all }\omega\in\Omega.$$
See \cite[Chapter 3, Section 2]{Talamanca_Picardello_1983} for details. Let $C(\Omega)$ be the space  of all continuous functions defined on the boundary $\Omega$. For $z\in \C$, we define the representations $\pi_z$ of $G$ on $C(\Omega)$ by the formula
$$\pi_z (g)\eta(\omega )= p^{1/2+ iz} (g\cdot o , \omega) \eta(g^{-1}\cdot\omega),\quad \text{for all }g\in G,\text{ for all }\omega\in \Omega.$$
It is clear that $ \pi_z=\pi_{z+\tau}$, where $\tau=2\pi/\log q$. We write $\mathbb{T}$ for the torus $\mathbb{R}/\tau\mathbb{Z}$, which we usually identify it with the interval $[-\tau/2,\tau/2)$. The elementary spherical function $\phi_z$ is now defined as
$$\phi_{z}(x)=\langle\pi_{z}(x)\mathbf{1},\mathbf{1}\rangle=\int\limits_{\Omega}p^{1/2+iz}(x,\omega)~d\nu(\omega),\quad\text{where }x\in\mathfrak{X}.$$
The following explicit formula of $\phi_{z}$ is well-known (see \cite[Theorem 2]{Talamanca_Alessandro_1982}):
\begin{equation}\label{eqsf}
	\phi_z(x)= \begin{cases}
		\vspace*{.2cm} \left(\frac{q-1}{q+1}|x|+1\right)q^{-|x|/2}&\forall z\in\ \tau\Z\\
		\vspace*{.2cm}\left(\frac{q-1}{q+1}|x|+1\right)q^{-|x|/2}(-1)^{|x|}&\forall z\in {\tau/2}+\tau\Z\\
		\mathbf{c}(z)q^{{(iz-1/2)}|x|}+\mathbf{c}(-z)q^{{(-iz-1/2)}|x|}&\forall z\in\C\setminus(\tau/2)\Z,
	\end{cases}
\end{equation}
where $\mathbf{c}$ is the meromorphic function given by
$$\mathbf{c}(z)=\frac{q^{1/2}}{q+1}\frac{q^{1/2+iz}-q^{-{1/2}-iz}}{q^{iz}-q^{-iz}}\quad\forall z\in \C \setminus (\tau/2)\Z.$$
We note that for every $x$ in $\mathfrak{X}$, the map $z\mapsto\phi_{z}(x)$ is an entire function. From the explicit formula above, it is also clear that $\phi_{z}$ is a radial function which satisfies $\phi_{z}=\phi_{-z}=\phi_{z+\tau}$ for every $z$ in $\mathbb{C}$.

The spherical transform $\widehat{f}$ of a finitely supported radial function $f$ on $\mathfrak{X} $ is defined by the formula
$$\widehat{f}(z)=\sum\limits_{x\in\mathfrak{X}}f(x)\phi_z(x),\quad \text{where } z\in\C.$$
One sees immediately that $z\mapsto\widehat{f}(z)$ is an entire function. Moreover, the symmetry property and the $\tau$-periodicity of $\phi_{z}$ implies that $\widehat{f}$ is even and $\tau$-periodic on $\mathbb{C}$. Following \cite{Meda_Stefano_2019}, we say that a holomorphic function $\eta$ defined on $S^{\circ}_{p}$ is Weyl-invariant if $\eta(z)=\eta(-z)$ and $\eta(z)=\eta(z+\tau)$ for all $z$ in $S^{\circ}_{p}$.

The Helgason-Fourier transform $\widetilde{f} $ of a finitely supported function $f$ on $\mathfrak{X}$ is a function on $\C\times\Omega$ defined by the formula
$$\widetilde{f} (z,\omega)=\sum\limits_{x\in\mathfrak{X}}f(x)p^{1/2+iz}(x,\omega).$$
It is clear that $\widetilde{f}(z,\omega)=\widetilde{f}(z+\tau,\omega)$ for every $z$ in $\mathbb{C}$. A simple computation reveals that if $f$ is radial, its Helgason Fourier transform $\widetilde{f}$ reduces to the spherical transform $\widehat{f}$. We conclude this section by stating the inversion formula for the Helgason-Fourier transform on $\mathfrak{X}$. For details, we refer to \cite[Chapter 3, Section IV and Chapter 5, Section IV]{Talamanca_Picardello_1983} and \cite[Chapter II, Section 6]{Alessandro_Nebbia}.

\begin{theorem}
If $f$ is a finitely supported function on $\mathfrak{X}$, then 
$$ f(x) = c_G\int\limits_{\T}\int\limits_{\Omega}p^{1/2-is}(x,\omega)  \widetilde{f} (s,\omega) |c(s)|^{-2} d\nu(\omega) ds,\quad\text{for all }x\in\mathfrak{X},$$
where $c_{G}=q\log q/4\pi(q+1)$.
\end{theorem}

\section{Analysis of pseudo-differential operators, \texorpdfstring{$1<p<2$}{1p2}}\label{PDO_on_tree_1<p<2}
In this section we prove our result concerning the $L^{p}$-boundedness of pseudo-differential operators on $\mathfrak{X}$, for $1<p<2$. We begin with a couple of preparatory lemmas. Both these lemmas are essentially known (see \cite[Lemma 2.2 $\&$ 2.3]{Meda_Stefano_2019} for details). However, it seems that there is a minor error in the statement of \cite[Lemma 2.2]{Meda_Stefano_2019}, which we have taken care of and provided a proof for the sake of completeness.

\begin{lemma}[{{\cite[Lemma 2.2]{Meda_Stefano_2019}}}]\label{nproperties}
Let $N$ be as in \eqref{groupn}. Then the following assertions are true.
\begin{enumerate}
\item For every $n\in N$, $d(o,n\cdot o)$ is either zero or a positive even number.
\item For every $n\in N$ and every $j\in\mathbb{Z}$ satisfying $j\leq d(o,n\cdot o)/2$,
		$$d(o,n\sigma^{j}\cdot o)=d(o,n\cdot o)-j.$$
\end{enumerate}
\end{lemma}
\begin{proof}
Before going into the details of the proof, let us first observe that the subgroup $N$ can also be written as
$$N=\bigcup\limits_{k\in\mathbb{Z}_{+}}N_{\omega^{0}_{k}},$$
where, for every $k$, $N_{\omega^{0}_{k}}=\{n\in G_{\omega_{0}}:n\cdot\omega^{0}_{k}=\omega^{0}_{k}\}$. Indeed, if $n$ is in $N$, then $n\cdot x=x$ for some $x$ in $\mathfrak{X}$. Since $n$ also fixes the equivalence class $[\omega^{+}_{0}]$, $n$ must fix every vertex of the infinite geodesic ray $\omega^{+}_{x}$ starting from $x$, and in the direction of $\omega^{+}_{0}$. By definition, the geodesic rays $\omega^{+}_{x}$ and $\omega^{+}_{0}$ are equivalent, and consequently there exists some $k_{0}$ in $\mathbb{Z}_{+}$ such that for all $k\geq k_{0}$, $n\cdot\omega^{0}_{k}=\omega^{0}_{k}$. Hence $n$ is in $N_{\omega^{0}_{k}}$ for all $k\geq k_{0}$. By applying a similar reasoning, one sees immediately that $N_{\omega^{0}_{k}}\subset N_{\omega^{0}_{k+1}}$ for all $k$. This further implies that $N$ can be written as the disjoint union
\begin{equation}\label{ndisjoint}
N=N_{\omega^{0}_{0}}\bigcup\limits_{k\in\mathbb{N}}\left(N_{\omega^{0}_{k}}\setminus N_{\omega^{0}_{k-1}}~\right).
\end{equation}

\noindent{\bf Proof of (1).} If $n$ is in $N$, by using (\ref{ndisjoint}) it follows that $n\in N_{\omega^{0}_{0}}$ or $n\in N_{\omega^{0}_{k}}\setminus N_{\omega^{0}_{k-1}}$ for some $k\in\mathbb{N}$. In the first case, $d(n\cdot o,o)=0$, whereas if $n\in N_{\omega^{0}_{k}}\setminus N_{\omega^{0}_{k-1}}$, then by using the facts $n\cdot \omega^{0}_{k}=\omega^{0}_{k}$ and $n\cdot \omega^{0}_{l}\neq\omega^{0}_{l}$ for all $0\leq l\leq k-1$, we get
$$d(n\cdot o,o)=d(n\cdot o,n\cdot \omega^{0}_{k})+d(n\cdot \omega^{0}_{k},o)=2~d(o,\omega^{0}_{k})=2k.$$

\noindent{\bf Proof of (2).} If $n\in N_{\omega^{0}_{0}}$, $d(o,n\cdot o)=0$ and hence for all $j\leq 0$,
$$d(o,n\sigma^{j}\cdot o)=d(n^{-1}\cdot o,\sigma^{j}\cdot o)=d(o,n\cdot o)-j.$$
Now suppose that $n\in N_{\omega^{0}_{k}}\setminus N_{\omega^{0}_{k-1}}$ for some $k\in\mathbb{N}$. Then $d(o,n\cdot o)=2k$ and consequently, for all $j\leq k$, we get
\begin{align*}
d(o,n\sigma^{j}\cdot o)&=d(o,n\sigma^{k}\cdot o)+d(n\sigma^{k}\cdot o,n\sigma^{j}\cdot o)\\
&=d(o,n\cdot \omega^{0}_{k})+d(n\cdot \omega^{0}_{k},n\cdot \omega^{0}_{j})\\
&=d(o,\omega^{0}_{k})+d(\omega^{0}_{k},\omega^{0}_{j})\\
&=2k-j,
\end{align*}
which establishes our claim.
\end{proof}

\begin{remark}
As mentioned before, it seems that Lemma \ref{nproperties} will not hold for all $n\in N$ and for all $j>d(o,n\cdot o)/2$. Here we give a counterexample.  We choose $n\in N_{\omega^{0}_{k}}\setminus N_{\omega^{0}_{k-1}}$ for some $k$ positive. Then from the preceding arguments, we have $d(o,n\cdot o)=2k$ and also $n\in N_{\omega^{0}_{l}}$ for all $l\geq k$. Consequently, for all $j>k$, $d(o,n\sigma^{j}\cdot o)=j>2k-j=d(o,n\cdot o)-j$.
\end{remark}

\begin{lemma}[{{\cite[Lemma 2.3]{Meda_Stefano_2019}}}]\label{qpn}
For $p\in[1,2)$, define $Q_{p}:N\rightarrow\mathbb{R}$ by $Q_{p}(n)=q^{-|n\cdot o|/p}$. Then the function $n\mapsto |n\cdot o|^{l}\,Q_{p}(n)$ belongs to $L^{1}(N)$ for each non-negative integer $l$.
\end{lemma}

\subsection{Decomposition of the operator \texorpdfstring{$T_\Psi$}{tpsi}}
We now express $T_\Psi$ as an integral operator on $G$ and thereafter, we shall decompose its kernel into two parts. Let $f$ be a finitely supported function on $\mathfrak{X}$. Then (\ref{pdotree}) gives us
$$T_{\Psi}f(h)=c_{G}\int\limits_{\mathbb{T}}\int\limits_{\Omega}\Psi(h,s)\widetilde{f}(s,\omega)p^{1/2-is}(h,\omega)|\mathbf{c}(s)|^{-2}~d\nu(\omega)~ds,\quad h\in G.$$
Substituting the expression of $\widetilde{f}(s,\omega)$, applying Fubini's Theorem and using the expression of $\phi_{s}(g^{-1}h)$ from \cite[Page 55]{Alessandro_Nebbia}, we obtain
\begin{align*}
	T_{\Psi}f(h)&=c_{G}\int\limits_{G}f(g)\int\limits_{\mathbb{T}}\Psi(h,s)|\mathbf{c}(s)|^{-2}\left(\int\limits_{\Omega}p^{1/2-is}(h,\omega)p^{1/2+is}(g,\omega)~d\nu(\omega)\right)ds~dg\\
	&=c_{G}\int\limits_{G}f(g)\left(\int\limits_{\mathbb{T}}\Psi(h,s)\phi_{s}(g^{-1}h)|\mathbf{c}(s)|^{-2}~ds\right)dg\\
	&=\int\limits_{G}f(g)K(h,g)~dg,
\end{align*}
where
\begin{equation}\label{pdokernel}
	K(h,g)=c_{G}\int\limits_{\mathbb{T}}\Psi(h,s)\phi_{s}(g^{-1}h)|\mathbf{c}(s)|^{-2}~ds.
\end{equation}
Using Theorem \ref{iwasawa} with $g=n\sigma^{j}k_{1}$ and $h=m\sigma^{l}k_{2}$, we can write
\begin{equation}\label{eqn_1st_defn_T_Psi}
	\begin{aligned}
T_{\Psi}f(m\sigma^{l})&=\int\limits_{N}\sum\limits_{j\in\mathbb{Z}}f(n\sigma^{j})K(m\sigma^{l},n\sigma^{j})q^{-j}~dn\\
&=T_{\Psi}^{+}f(m\sigma^{l})+T_{\Psi}^{-}f(m\sigma^{l})\quad\text{(say)},
\end{aligned}
\end{equation}
where the operators $T_{\Psi}^{\pm}f$ are defined by the formulae
\begin{equation}\label{tpsipm}
	T_{\Psi}^{\pm}f(m\sigma^{l})=\int\limits_{N}\sum\limits_{j\in\mathbb{Z}}f(n\sigma^{j})K(m\sigma^{l},n\sigma^{j})\chi_{\pm}(l-j)q^{-j}~dn,
\end{equation}
and $\chi_{\pm}$ are functions on $\mathbb{Z}$ given by
$$\chi_{+}(l-j)=\chi_{[0,\infty)}(l-j)\quad\text{and}\quad\chi_{-}(l-j)=\chi_{(-\infty,-1]}(l-j).$$
Consequently, the $L^{p}$-boundedness of $T_{\Psi}$ follows from that of the operators $T_{\Psi}^{\pm}$, which we shall take up separately in the following theorems.

In the case of multipliers on $\mathfrak{X}$, Meda et al. \cite{Meda_Stefano_2019} decomposed $T_m$ (defined by \eqref{multiplierx}) into the operators $T_{m}^{\pm}$, which are of convolution type with $K$-biinvariant functions $ K_m^{\pm}$ (say). In order to get the boundedness of $T_{m}^{-}$, they first proved a general transference result for convolution operators (see \cite[Theorem 3.3]{Meda_Stefano_2019}) and then used the estimate of the kernel $K_{m}^{-}$. On the other hand, the boundedness of $T_{m}^{+}$ was an outcome of a basic convolution type inequality given in \cite[Corollary 20.14 (ii) $\&$ (iv)]{MR551496}. However, in the case of pseudo-differential operators, we don't have the privilege to use the above methods directly. This is the crucial difference between the multiplier case and our situation. We prove the boundedness of $T_{\Psi}^{-}$ by establishing a connection with pseudo-differential operators on $\mathbb{Z}$. In the case of $T_{\Psi}^{+}$, we shall broadly follow the approach of Ionescu \cite{Ionescu_2002} for noncompact symmetric spaces.

\begin{theorem}\label{tpsiminus}
Let $1<p<2$. Suppose that $\Psi$ satisfies the hypothesis of Theorem \ref{mainresult1} and $T_{\Psi}^{-}$ be as in \eqref{tpsipm}. Then $T_{\Psi}^{-}$ is bounded from $L^{p}(NA)$ to itself. Moreover, there exists a constant $C_{p}>0$ such that
$$\|T_{\Psi}^{-}f\|_{L^{p}(NA)}\leq C_{p}\left(\|\Psi\|_{L^{\infty}(\mathfrak{X}\times S^{\circ}_{p})}+\sup\limits_{m\in N}\vertiii{\Psi_{-\delta_{p}}(m\cdot,\cdot)}_{p}\right) \|f\|_{L^{p}(NA)},$$
for all $f\in L^{p}(NA)$.
\end{theorem}
\begin{proof}
We will prove this theorem in the following steps:

\noindent \textbf{STEP I : Analysis of the Kernel.} We recall from (\ref{pdokernel}) that
$$K(h,g)=c_{G}\int\limits_{\mathbb{T}}\Psi(h,s)\phi_{s}(g^{-1}h)|\mathbf{c}(s)|^{-2}~ds.$$
Putting the explicit expression of $\phi_{s}$ from (\ref{eqsf}) and using the fact that $|\mathbf{c}(s)|^{2}=\mathbf{c}(s)\mathbf{c}(-s)$ for all $s\in\T$, we obtain
$$K(h,g)=c_{G}\int\limits_{\mathbb{T}}\Psi(h,s)q^{(is-1/2)|g^{-1}h\cdot o|}\mathbf{c}(-s)^{-1}~ds+c_{G}\int\limits_{\mathbb{T}}\Psi(h,s)q^{(-is-1/2)|g^{-1}h\cdot o|}\mathbf{c}(s)^{-1}~ds.$$
After a change of variable $s\mapsto -s$ in the second integral and using the Weyl-invariance of the function $\Psi(h,\cdot)$, we get
\begin{equation}\label{eqn_expression_of_K_1st}
	K(h,g)=C\int\limits_{\mathbb{T}}\Psi(h,s)q^{(is-1/2)|g^{-1}h\cdot o|}\mathbf{c}(-s)^{-1}~ds.
\end{equation}
We note that the above integrand is holomorphic function on $S^{\circ}_{p}$. Therefore applying the Cauchy's integral theorem on the closed rectangle
\begin{multline*}
		\Gamma(z)=\{z\in\mathbb{C}:\Im z=0,-\tau/2\leq \Re z\leq\tau/2\}\cup\{z\in\mathbb{C}: \Re z=\tau/2,0\leq \Im z\leq\delta_{r}\}\\
		\cup\{z\in\mathbb{C}:\Im z=\delta_{r},\tau/2\leq \Re z\leq-\tau/2\}\cup\{z\in\mathbb{C}:\Re z=-\tau/2,\delta_{r}\leq \Im z\leq 0\},
\end{multline*}
	it follows that for every $r\in(p,2)$,
	$$K(h,g)=C\int\limits_{\mathbb{T}}\Psi(h,s+i\delta_{r})q^{(is-1/r)|g^{-1}h\cdot o|}\mathbf{c}(-s-i\delta_{r})^{-1}~ds.$$
Using the dominated convergence theorem and letting $r\rightarrow p$, we finally get
	\begin{align*}
		K(h,g)&=C\int\limits_{\mathbb{T}}\Psi(h,s+i\delta_{p})q^{(is-1/p)|g^{-1}h\cdot o|}\mathbf{c}(-s-i\delta_{p})^{-1}~ds\\
		&=C~q^{-|g^{-1}h\cdot o|/p}\int\limits_{\mathbb{T}}\Psi_{\delta_{p}}(h,s)q^{is|g^{-1}h\cdot o|}\mathbf{c}_{-\delta_{p}}(-s)^{-1}~ds.
	\end{align*}
Plugging in the above expression in (\ref{tpsipm}), we obtain
\begin{multline*}
T_{\Psi}^{-}f(m\sigma^{l})=C\int\limits_{N}\sum\limits_{j\in\mathbb{Z}}f(n\sigma^{j})q^{-|\sigma^{-j}n^{-1}m\sigma^{l}\cdot o|/p}\\
\cdot\left(\int\limits_{\mathbb{T}}\Psi_{\delta_{p}}(m\sigma^{l},s)q^{is|\sigma^{-j}n^{-1}m\sigma^{l}\cdot o|}\mathbf{c}_{-\delta_{p}}(-s)^{-1}~ds\right)\chi_{-}(l-j)q^{-j}~dn.
\end{multline*}
	The change of variable $n\mapsto m^{-1}n$ implies that
	\begin{multline*}
		T_{\Psi}^{-}f(m\sigma^{l})=C\int\limits_{N}\sum\limits_{j\in\mathbb{Z}}f(mn\sigma^{j})q^{-|\sigma^{-j}n^{-1}\sigma^{l}\cdot o|/p}\\
		\cdot\left(\int\limits_{\mathbb{T}}\Psi_{\delta_{p}}(m\sigma^{l},s)q^{is|\sigma^{-j}n^{-1}\sigma^{l}\cdot o|}\mathbf{c}_{-\delta_{p}}(-s)^{-1}~ds\right)\chi_{-}(l-j)q^{-j}~dn.
	\end{multline*}
After a change of variable $n\mapsto \sigma^{-j}n\sigma^{j}$ and using \eqref{modularfunction}, we have
	\begin{multline*}
		T_{\Psi}^{-}f(m\sigma^{l})=C\int\limits_{N}\sum\limits_{j\in\mathbb{Z}}f(m\sigma^{j}n)q^{-|n^{-1}\sigma^{l-j}\cdot o|/p}\\
		\cdot\left(\int\limits_{\mathbb{T}}\Psi_{\delta_{p}}(m\sigma^{l},s)q^{is|n^{-1}\sigma^{l-j}\cdot o|}\mathbf{c}_{-\delta_{p}}(-s)^{-1}~ds\right)\chi_{-}(l-j)~dn.
	\end{multline*}
Since $\chi_{-}(l-j)=0$, whenever $l\geq j$, therefore using Lemma \ref{nproperties} $(2)$, we get
	\begin{multline}\label{kappamkappamn}
		T_{\Psi}^{-}f(m\sigma^{l})=C\int\limits_{N}\sum\limits_{j\in\mathbb{Z}}f(m\sigma^{j}n)q^{-|n^{-1}\cdot o|/p}q^{(l-j)/p}\\
		\cdot\left(\int\limits_{\mathbb{T}}\Psi_{\delta_{p}}(m\sigma^{l},s)q^{is(|n^{-1}\cdot o|-(l-j))}\mathbf{c}_{-\delta_{p}}(-s)^{-1}~ds\right)\chi_{-}(l-j)~dn.
	\end{multline}
For fixed $m,n\in N$, let us introduce the notation
\begin{equation}\label{kappapdoz}
\begin{aligned}
\kappa_{m}(l,|n^{-1}\cdot o|-(l-j))&=\int\limits_{\mathbb{T}}\Psi_{\delta_{p}}(m\sigma^{l},s)q^{is(|n^{-1}\cdot o|-(l-j))}\mathbf{c}_{-\delta_{p}}(-s)^{-1}~ds\\
&=\kappa_{m,n}(l,l-j).
\end{aligned}
\end{equation}
We shall use both the conventions $\kappa_{m}(\cdot,\cdot)$ and $\kappa_{m,n}(\cdot,\cdot)$ for the above integral as and when necessary. Consequently, the expression of $T_{\Psi}^{-}f$ in (\ref{kappamkappamn}) takes the form
	$$T_{\Psi}^{-}f(m\sigma^{l})=C\int\limits_{N}\sum\limits_{j\in\mathbb{Z}}f(m\sigma^{j}n)q^{-|n^{-1}\cdot o|/p}q^{(l-j)/p}\kappa_{m,n}(l,l-j)\chi_{-}(l-j)~dn.$$
We recall from Lemma \ref{qpn} that $Q_{p}(n)=q^{-|n\cdot o|/p}$. Consequently,
	\begin{multline*}
		\|T_{\Psi}^{-}f\|_{L^{p}(NA)}=C\left\|~\left\|~\int\limits_{N}Q_{p}(n^{-1})\sum\limits_{j\in\mathbb{Z}}f(m\sigma^{j}n)q^{-j/p}\kappa_{m,n}(l,l-j)\right.\right.\\
		\left.\left.\cdot\chi_{(-\infty,-1]}(l-j)~dn\right\|_{L^{p}(\mathbb{Z},l)}~\right\|_{L^{p}(N,\,dm)},
	\end{multline*}
and using Minkowski's inequality, we deduce that
	\begin{multline}\label{minkowskifinal}
		\|T_{\Psi}^{-}f\|_{L^{p}(NA)}\leq C\int\limits_{N}Q_{p}(n^{-1})\left\|~\left\|\sum\limits_{j\in\mathbb{Z}}f(m\sigma^{j}n)q^{-j/p}\kappa_{m,n}(l,l-j)\right.\right.\\
		\left.\left.\cdot\chi_{(-\infty,-1]}(l-j)\right\|_{L^{p}(\mathbb{Z},l)}~\right\|_{L^{p}(N,\,dm)}~dn.
	\end{multline}
For a fixed $m\in N$, let us define the operator
$$\mathcal{B}_n\theta(l)=\sum\limits_{j\in\mathbb{Z}}\theta(j)\kappa_{m,n}(l,l-j)\chi_{(-\infty,-1]}(l-j),\quad\text{for all }l\in\mathbb{Z}.$$
We claim that $ \mathcal{B}_n $ defines a bounded pseudo-differential operator from $L^{p}(\mathbb{Z})$ to itself. Moreover, we will show there exists a constant $ C_{p}>0 $, independent of $m$ and $n$, such that the following holds:
	\begin{equation}\label{eqn_bddness_of_T_1}
		\|\mathcal{B}_n \theta\|_{L^p(\Z)} \leq C_{p} (1+|n^{-1}\cdot o|)\|\theta\|_{L^p(\Z)}, \quad \text{ for all } \theta \in L^p(\Z).
	\end{equation}
Assuming the claim above, we complete the proof of Theorem \ref{tpsiminus}. Plugging in the estimate \eqref{eqn_bddness_of_T_1} with $\theta(j)=f(m\sigma^{j}n)q^{-j/p}$, into \eqref{minkowskifinal}, we derive that
\begin{align*} 
\|T_{\Psi}^{-}f\|_{L^{p}(NA)}&\leq C_{p}\int\limits_{N} (1+|n^{-1}\cdot o|)Q_{p}(n^{-1})\left\| \left( \sum\limits_{j\in\mathbb{Z}}|f(m\sigma^{j}n)|^pq^{-j} \right)^{1/p} \right\|_{L^{p}(N,\,dm)}dn\\
& = C_{p} \int\limits_N   \left(  \int\limits_N \sum\limits_{j\in\mathbb{Z}}|f(m\sigma^{j}n\sigma^{-j}\sigma^{j})|^pq^{-j} dm \right)^{1/p} (1+|n^{-1}\cdot o|)Q_{p}(n^{-1})~dn\\
&\leq C _{p}\|f\|_{L^p(NA)},
\end{align*}
where in the last inequality, we have used Theorem \ref{iwasawa} and Lemma \ref{qpn}. This proves Theorem \ref{tpsiminus}, modulo the claim in \eqref{eqn_bddness_of_T_1}.

\noindent \textbf{STEP II : Connection with pseudo-differential operators on $ \Z. $} Here we  prove \eqref{eqn_bddness_of_T_1}. We recall from (\ref{kappapdoz}) an alternative definition  of $\kappa_{m,n}(\cdot,\cdot)$, and observe that the $L^p(\Z)$ operator norm of $\mathcal{B}_n$ is equal to that of $\mathcal{B}^{\prime}_n$, defined by
	$$\mathcal{B}^{\prime}_n\theta(l):=\sum\limits_{j\in\mathbb{Z}}\theta(j)\kappa_{m}(-l,l-j)\chi_{[1+|n^{-1}\cdot o|,\infty)}(l-j),\quad\text{for all }l\in\mathbb{Z}.$$
In fact, an explicit calculation yields
\begin{align*}\allowdisplaybreaks
		\vertiii{\mathcal{B}_n}_{p}&=\sup\limits_{\|\theta\|_{L^{p}(\mathbb{Z})}=1}\left(\sum\limits_{l\in\mathbb{Z}}\left|\sum\limits_{j\in\mathbb{Z}}\theta(j)\kappa_{m}(l,|n^{-1}\cdot o|-(l-j))\chi_{(-\infty,-1]}(l-j)\right|^{p}~\right)^{1/p}\\
		&=\sup\limits_{\|\theta\|_{L^{p}(\mathbb{Z})}=1}\left(\sum\limits_{l\in\mathbb{Z}}\left|\sum\limits_{j\in\mathbb{Z}}\theta(j-|n^{-1}\cdot o|)\kappa_{m}(l,j-l)\chi_{(-\infty,-1-|n^{-1}\cdot o|]}(l-j)\right|^{p}~\right)^{1/p}\\
		&=\sup\limits_{\|\theta\|_{L^{p}(\mathbb{Z})}=1}\left(\sum\limits_{l\in\mathbb{Z}}\left|\sum\limits_{j\in\mathbb{Z}}\theta(j-|n^{-1}\cdot o|)\kappa_{m}(l,j-l)\chi_{[1+|n^{-1}\cdot o|,\infty)}(j-l)\right|^{p}~\right)^{1/p}\\
		&=\sup\limits_{\|\theta\|_{L^{p}(\mathbb{Z})}=1}\left(\sum\limits_{l\in\mathbb{Z}}\left|\sum\limits_{j\in\mathbb{Z}}\theta(j-|n^{-1}\cdot o|)\kappa_{m}(-l,j+l)\chi_{[1+|n^{-1}\cdot o|,\infty)}(j+l)\right|^{p}~\right)^{1/p}\\
		&=\sup\limits_{\|\theta\|_{L^{p}(\mathbb{Z})}=1}\left(\sum\limits_{l\in\mathbb{Z}}\left|\sum\limits_{j\in\mathbb{Z}}\theta(-j-|n^{-1}\cdot o|)\kappa_{m}(-l,l-j)\chi_{[1+|n^{-1}\cdot o|,\infty)}(l-j)\right|^{p}~\right)^{1/p}\\
&=\sup\limits_{\|\theta\|_{L^{p}(\mathbb{Z})}=1}\left(\sum\limits_{l\in\mathbb{Z}}\left|\sum\limits_{j\in\mathbb{Z}}\theta(j)\kappa_{m}(-l,l-j)\chi_{[1+|n^{-1}\cdot o|,\infty)}(l-j)\right|^{p}~\right)^{1/p}\\
&=\vertiii{\mathcal{B}^{\prime}_n}_{p}.
	\end{align*}
Thus, it is enough to  prove that $\mathcal{B}^{\prime}_n$ satisfies \eqref{eqn_bddness_of_T_1}. We can write 
\begin{equation}\label{eqn_decompo_ofA_n_as_I}
		 \mathcal{B}^{\prime}_n\theta (l) = \mathcal{I}_+\theta(l)-\mathcal{I}_- \theta(l)-\mathcal{I}_n\theta(l), \quad \text{for all } l \in \Z,
\end{equation} 
where $\mathcal{I}_+$, $\mathcal{I}_-$ and $\mathcal{I}_n$ are pseudo-differential operators on $\Z$ defined by
\begin{align*}
\mathcal{I}_+ \theta(l)&=\sum\limits_{j\in\mathbb{Z}}\theta(j)\kappa_{m}(-l,l-j),\\
\mathcal{I}_-\theta(l)&= \sum\limits_{j\in\mathbb{Z}}\theta(j)\kappa_{m}(-l,l-j)\chi_{(-\infty,-1]}(l-j),\\
\mathcal{I}_n\theta(l)&= \sum\limits_{j\in\mathbb{Z}}\theta(j)\kappa_{m}(-l,l-j)\chi_{[0,|n^{-1}\cdot o|]}(l-j).
\end{align*}
First we shall estimate $\mathcal{I}_+$. Using the explicit expression of $\kappa_{m}(\cdot,\cdot)$ from (\ref{kappapdoz}), it follows that
	\begin{align*}
		\mathcal{I}_+\theta(l)&=\sum\limits_{j\in\mathbb{Z}}\theta(j)\left(\int\limits_{\mathbb{T}}\Psi_{\delta_{p}}(m\sigma^{-l},s)q^{is(l-j)}\mathbf{c}_{-\delta_{p}}(-s)^{-1}~ds\right)\\
		&=\int\limits_{\mathbb{T}}\Psi_{\delta_{p}}(m\sigma^{-l},s)\mathcal{F}\theta(s)\mathbf{c}_{-\delta_{p}}(-s)^{-1}q^{isl}~ds.
	\end{align*}
	The change of variable $s\mapsto -s$ and the Weyl-invariance of $\Psi$ yields
	$$\mathcal{I}_+\theta(l)=\int\limits_{\mathbb{T}}\Psi_{-\delta_{p}}(m\sigma^{-l},s)\mathcal{F}\theta(-s)\mathbf{c}_{-\delta_{p}}(s)^{-1}q^{-isl}~ds.$$
	Hence the operator norm of $\mathcal{I}_+$ on $L^{p}(\mathbb{Z})$ becomes
	\begin{align*}
		\vertiii{\mathcal{I}_+}_{p}&=\sup\limits_{\|\theta\|_{L^{p}(\mathbb{Z})}=1}\left(\sum\limits_{l\in\mathbb{Z}}\left|\int\limits_{\mathbb{T}}\Psi_{-\delta_{p}}(m\sigma^{-l},s)\mathcal{F}\theta(-s)\mathbf{c}_{-\delta_{p}}(s)^{-1}q^{-isl}~ds\right|^{p}~\right)^{1/p}\\
		&=\sup\limits_{\|\theta\|_{L^{p}(\mathbb{Z})}=1}\left(\sum\limits_{l\in\mathbb{Z}}\left|\int\limits_{\mathbb{T}}\Psi_{-\delta_{p}}(m\sigma^{l},s)\mathcal{F}\theta(-s)\mathbf{c}_{-\delta_{p}}(s)^{-1}q^{isl}~ds\right|^{p}~\right)^{1/p}.
	\end{align*}
We observe that the inner integral above defines a pseudo-differential operator on $\mathbb{Z}$ as in \eqref{pdoz1} with symbol $\psi(l,s)=\Psi(m\sigma^{l},s-i\delta_{p})$. Since $\Psi_{-\delta_{p}}(m\cdot,\cdot)$ defines a bounded pseudo-differential operator on $L^{p}(\mathbb{Z})$, the above expression gives
	\begin{equation}\label{t3final}
		\vertiii{\mathcal{I}_+}_{p}\leq \vertiii{\Psi_{-\delta_{p}}(m\cdot,\cdot)}_{p}\left(\sup\limits_{\|\theta\|_{L^{p}(\mathbb{Z})}=1} \|\mathcal{F}^{-1}\left(\mathcal{F}\theta(-\cdot)\mathbf{c}_{-\delta_{p}}(\cdot)^{-1}\right)\|_{L^{p}(\mathbb{Z})} \right).
	\end{equation}
Define $\theta^{\#}(j)=\theta(-j)$, where $j\in\mathbb{Z}$. A simple calculation yields that  $\mathcal{F}(\theta^{\#})(s)=\mathcal{F}\theta(-s)$. Implementing this fact and using Young's inequality in (\ref{t3final}) we   obtain
	\begin{align*}
		\vertiii{\mathcal{I}_+}_{p}&\leq \vertiii{\Psi_{-\delta_{p}}(m\cdot,\cdot)}_{p}\left(\sup\limits_{\|\theta\|_{L^{p}(\mathbb{Z})}=1} \|\theta^{\#}\ast_{\mathbb{Z}}\mathcal{F}^{-1}\left(\mathbf{c}_{-\delta_{p}}(\cdot)^{-1}\right)\|_{L^{p}(\mathbb{Z})} \right)\\
		&\leq \vertiii{\Psi_{-\delta_{p}}(m\cdot,\cdot)}_{p}\left(\sup\limits_{\|\theta\|_{L^{p}(\mathbb{Z})}=1} \|\theta^{\#}\|_{L^{p}(\mathbb{Z})}\|\mathcal{F}^{-1}\left(\mathbf{c}_{-\delta_{p}}(\cdot)^{-1}\right)\|_{L^{1}(\mathbb{Z})} \right).
	\end{align*}
	Since $s\mapsto\mathbf{c}(-s-i\delta_{p})^{-1}$ is a smooth function on $\mathbb{T}$, we finally get
\begin{equation}\label{eqn_op_norm_est_of_I_0}
\vertiii{\mathcal{I}_+}_{p}\leq C_{p}\left(\sup\limits_{m\in N}\vertiii{\Psi_{-\delta_{p}}(m\cdot,\cdot)}_{p}\right).
	\end{equation}
	Next, we consider the operator
	$$\mathcal{I}_-\theta(l)=\sum\limits_{j\in\mathbb{Z}}\theta(j)\kappa_{m}(-l,l-j)\chi_{(-\infty,-1]}(l-j).$$
Recalling the expression of $\kappa_{m}(\cdot,\cdot)$ from (\ref{kappapdoz}), it follows that
	$$\kappa_{m}(-l,l-j)=\int\limits_{\mathbb{T}}\Psi_{\delta_{p}}(m\sigma^{-l},s)q^{is(l-j)}\mathbf{c}_{-\delta_{p}}(-s)^{-1}~ds.$$
We observe that the integrand above  is holomorphic on the strip
	$$\textbf{S}^{-}_{2\delta_{p}}:=\{z\in\mathbb{C}:  -2\delta_{p}<\Im z<0\}.$$
	Thus, using Cauchy's integral theorem and the dominated convergence theorem, we get
	$$\kappa_{m}(-l,l-j)=q^{2\delta_{p}(l-j)}\int\limits_{\mathbb{T}}\Psi_{-\delta_{p}}(m\sigma^{-l},s)q^{is(l-j)}\mathbf{c}_{\delta_{p}}(-s)^{-1}~ds.$$
Taking modulus on both the sides of the above expression, we get the pointwise estimate
	$$|\kappa_{m}(-l,l-j)|\leq C_{p}\,q^{2\delta_{p}(l-j)}\|\Psi\|_{L^{\infty}(\mathfrak{X}\times S_{p})},\quad\text{for all }l,j\in\mathbb{Z}.$$
This in turn implies that
	\begin{align}\label{t4final}
		\vertiii{\mathcal{I}_-}_{p}&\leq\sup\limits_{\|\theta\|_{L^{p}(\mathbb{Z})}=1}\left(\sum\limits_{l\in\mathbb{Z}}\left(\sum\limits_{j\in\mathbb{Z}}|\theta(j)|~|\kappa_{m}(-l,l-j)|~\chi_{(-\infty,-1]}(l-j)\right)^{p}~\right)^{1/p}\nonumber\\
		&\leq C_{p}~\|\Psi\|_{L^{\infty}(\mathfrak{X}\times S_{p})}\sup\limits_{\|\theta\|_{L^{p}(\mathbb{Z})}=1}\left(\sum\limits_{l\in\mathbb{Z}}\left(\sum\limits_{j\in\mathbb{Z}}|\theta(j)|~q^{2\delta_{p}(l-j)}~\chi_{(-\infty,-1]}(l-j)\right)^{p}~\right)^{1/p}\nonumber\\
		&=C_{p}~\|\Psi\|_{L^{\infty}(\mathfrak{X}\times S_{p})}\sup\limits_{\|\theta\|_{L^{p}(\mathbb{Z})}=1}\left(\sum\limits_{l\in\mathbb{Z}}\left(\sum\limits_{j\in\mathbb{Z}}|\theta(l-j)|~q^{2j\delta_{p}}~\chi_{(-\infty,-1]}(j)\right)^{p}~\right)^{1/p}\nonumber\\
		&\leq C_{p}~\|\Psi\|_{L^{\infty}(\mathfrak{X}\times S_{p})}\sup\limits_{\|\theta\|_{L^{p}(\mathbb{Z})}=1}\left(\sum\limits_{j\in\mathbb{Z}}\left(q^{2j\delta_{p}}~\chi_{(-\infty,-1]}(j)\left(\sum\limits_{l\in\mathbb{Z}}|\theta(l-j)|^{p}\right)^{1/p}\right)\right)\nonumber\\
		&\leq C_{p}~\|\Psi\|_{L^{\infty}(\mathfrak{X}\times S_{p})}.
\end{align}
Finally, we investigate the operator norm of
	$$\mathcal{I}_n\theta(l)=\sum\limits_{j\in\mathbb{Z}}\theta(j)\kappa_{m}(-l,l-j)\chi_{[0,|n^{-1}\cdot o|]}(l-j).$$
The expression of $\kappa_{m}(\cdot,\cdot)$ from (\ref{kappapdoz}) gives us the trivial estimate
	$$|\kappa_{m}(-l,l-j)|\leq C_{p}\|\Psi\|_{L^{\infty}(\mathfrak{X}\times S_{p})},\quad\text{for all }l,j\in\mathbb{Z}.$$
Consequently, as in \eqref{t4final} we have that
	\begin{equation}\label{t5final}
		\vertiii{\mathcal{I}_n}_{p}\leq C_{p}(1+|n^{-1}\cdot o|)\,\|\Psi\|_{L^{\infty}(\mathfrak{X}\times S_{p})}.
	\end{equation}
Plugging in the estimates \eqref{eqn_op_norm_est_of_I_0}, \eqref{t4final} and \eqref{t5final} of the operators $ \mathcal{I}_+, \mathcal{I}_-$, and $\mathcal{I}_n $ respectively,  into \eqref{eqn_decompo_ofA_n_as_I}, we get the desired  claim \eqref{eqn_bddness_of_T_1}. This also concludes the proof of Theorem \ref{tpsiminus}.
\end{proof}
The remainder of this section will be devoted to the proof of $ L^p(NA) $-boundedness of the operator   $ T_\Psi ^+$, for $1<p<2$. Consequently, we will be able to conclude that $ T_\Psi $ is bounded from $ L^p(\mathfrak{X}) $ to itself.  
   
\begin{theorem}\label{thm:tpsi_plus_bdd}
	Let $ 1<p<2$. Suppose that $ \Psi $ satisfies the hypothesis (1) and (2) of Theorem \ref{mainresult1} and  $T_{\Psi}^{+}$ be as in \eqref{tpsipm}. Then $T_{\Psi}^{+}$   is bounded from $L^{p}(NA)$ to itself. Moreover, there exists a constant $ C_{p}>0 $  such that  
	$$\|T_{\Psi}^{+}f\|_{L^{p}(NA)} \leq C_{p} \| \Psi\|_{L^\infty(\mathfrak{X}\times S^\circ_p)}   \|f\|_{L^{p}(NA)},\quad\text{for all }f\in L^{p}(NA).$$ 
\end{theorem}
\begin{proof}
	Recalling the definition of $T_{\Psi}^{+}  $ from \eqref{tpsipm}, we have
	$$T^+_{\Psi}f(m\sigma^l) =\int\limits_{N}\sum\limits_{j\in\mathbb{Z}}f(n\sigma^{j})K(m\sigma^{l},n\sigma^{j})\chi_{+}(l-j)q^{-j}~dn,$$
which after substituting the expression of $ K(m\sigma^{l},n\sigma^{j})$ from \eqref{eqn_expression_of_K_1st} becomes 
$$T^+_\Psi f(m\sigma^l) = C  \int\limits_{N}\sum\limits_{j\in\mathbb{Z}}f(n\sigma^{j})  \left(\int\limits_{\mathbb{T}}\Psi(m\sigma^l,s) \mathbf{c}(-s)^{-1}\,q^{(is-1/2)|\sigma^{-j}n^{-1}m\sigma^{l}\cdot o|}ds\right)\chi_{+}(l-j)q^{-j} dn.$$
Using the same argument as in  Theorem \ref{tpsiminus}, we move the contour of integration in the inner integral from $ \T $  to $ \T+i\delta_{p} $  and obtain
\begin{multline}\label{eqn:equi_expn_of_T_psi_+}
T_{\Psi}^{+}f(m\sigma^{l})=C\int\limits_{N}\sum\limits_{j\in\mathbb{Z}}f(n\sigma^{j})q^{-|\sigma^{-j}n^{-1}m\sigma^{l}\cdot o|/p}\\
\cdot \left(\int\limits_{\mathbb{T}}\Psi_{\delta_{p}}(m\sigma^{l},s) \mathbf{c}_{-\delta_{p}}(-s)^{-1} q^{is|\sigma^{-j}n^{-1}m\sigma^{l}\cdot o|}~ds\right)\chi_{+}(l-j)q^{-j}~dn.
\end{multline}
To get the desired result, it suffices to prove that for any compactly supported functions $ f,\varphi : N A\rightarrow \C$, one has
$$\left| \left\langle T_{\Psi}^{+}f, \varphi\right\rangle\right|=  \left| \int\limits_{N}\sum\limits_{l\in\mathbb{Z} }   T_{\Psi}^{+}f(m\sigma^{l}) \varphi(m\sigma^l) q^{-l} dm\right|\leq C_p\, \|f\|_{L^p(NA)}  \|\varphi\|_{L^{p'}(NA)}.$$
Plugging in the expression of $T_{\Psi}^{+}f  $ from \eqref{eqn:equi_expn_of_T_psi_+}, we get
\begin{multline*}
\left\langle T_{\Psi}^{+}f, \varphi\right\rangle= C\int\limits_{N}  \int\limits_{N} \sum\limits_{l\in\mathbb{Z} } \sum\limits_{j\in\mathbb{Z} }   f(n\sigma^{j})  \varphi(m\sigma^l)\chi_{+}(l-j) q^{-|\sigma^{-j}n^{-1}m\sigma^{l}\cdot o|/p} q^{-j-l}\\
\cdot \left(\int\limits_{\mathbb{T}}\Psi_{\delta_{p}}(m\sigma^{l},s) \mathbf{c}_{-\delta_{p}}(-s)^{-1} q^{is|\sigma^{-j}n^{-1}m\sigma^{l}\cdot o|}~ds\right) \,dm\, dn,
\end{multline*}
which after a change of variable $m\mapsto n^{-1}m$ yields
\begin{multline*}
\left\langle T_{\Psi}^{+}f, \varphi\right\rangle=C \int\limits_{N}  \int\limits_{N} \sum\limits_{l\in\mathbb{Z} } \sum\limits_{j\in\mathbb{Z} } f(n\sigma^{j})\varphi(nm\sigma^l)\chi_{+}(l-j) q^{-|\sigma^{-j}m\sigma^{l}\cdot o|/p}q^{-j-l}\\
\cdot \left(\int\limits_{\mathbb{T}}\Psi_{\delta_{p}}(nm\sigma^{l},s)\mathbf{c}_{-\delta_{p}}(-s)^{-1}q^{is|\sigma^{-j}m\sigma^{l}\cdot o|}~ds\right)~dm~dn.
\end{multline*}
We recall that the map $m\mapsto \sigma^{-j} m \sigma^j$ is a dilation of $N$. Hence by using \eqref{modularfunction}, we get
\begin{multline*}
\left\langle T_{\Psi}^{+}f, \varphi\right\rangle= C \int\limits_{N}  \int\limits_{N} \sum\limits_{l\in\mathbb{Z} } \sum\limits_{j\in\mathbb{Z} }   f(n\sigma^{j})  \varphi(n\sigma^jm \sigma^{l-j})\chi_{+}(l-j) q^{-|m \sigma^{l-j}\cdot o|/p}q^{-l}\\
\cdot \left(\int\limits_{\mathbb{T}}\Psi_{\delta_{p}}(n\sigma^j m \sigma^{l-j},s)\mathbf{c}_{-\delta_{p}}(-s)^{-1} q^{is| m \sigma^{l-j}\cdot o|}~ds\right) ~dm~dn.
\end{multline*}
Next, by using Fubini's theorem and the change of variable $ l \mapsto l-j $, it follows that    
\begin{multline*}
	\left\langle T_{\Psi}^{+}f, \varphi\right\rangle= C \int\limits_{N}  \int\limits_{N} \sum\limits_{j\in\mathbb{Z} } \sum\limits_{l\in\mathbb{Z}_+ }   f(n\sigma^{j})  \varphi(n\sigma^jm \sigma^{l}) q^{-|m \sigma^{l }\cdot o|/p}q^{-l-j}\\
	\cdot \left(\int\limits_{\mathbb{T}}\Psi_{\delta_{p}}(n\sigma^j m \sigma^{l},s)\mathbf{c}_{-\delta_{p}}(-s)^{-1} q^{is| m \sigma^{l}\cdot o|}~ds\right) ~dm~dn.
\end{multline*}
Taking modulus on both the sides of the above expression and using the boundedness of $\Psi$ and $\mathbf{c}_{-\delta_{p}}(-s)^{-1}$ on $\mathfrak{X}\times S_{p}$ and $\T$ respectively, we deduce that
\begin{multline}\label{eqn:equi_expn_of_T_psi_+_2} 
\left|	\left\langle T_{\Psi}^{+}f, \varphi\right\rangle\right|\leq C_{p} \| \Psi\|_{L^\infty(\mathfrak{X}\times S_p)} \\
 \int\limits_{N}\int\limits_{N}\sum\limits_{j\in\mathbb{Z} } \sum\limits_{l\in\mathbb{Z}_+ }   |f(n\sigma^{j}) |  |\varphi(n\sigma^jm \sigma^{l})|  q^{-|m \sigma^{l}\cdot o|/p}q^{-l-j} \,dm\,dn.
\end{multline}
Now let us define,
	\begin{equation}\label{fphi}
			F(\sigma^j) =  \left[\int\limits_{N} \left|f(n\sigma^j)\right|^p dn  \right]^{1/p}, \quad \text{and} \quad
			\varPhi(\sigma^j) =  \left[\int\limits_{N} \left|\varphi(n\sigma^j)\right|^{p^\prime} dn  \right]^{1/p^\prime}.
	\end{equation}
Applying H\"older's inequality in \eqref{eqn:equi_expn_of_T_psi_+_2} and using \eqref{fphi}, we get
$$\left|	\left\langle T_{\Psi}^{+}f, \varphi\right\rangle\right|\leq C_{p} \| \Psi\|_{L^\infty(\mathfrak{X}\times S_p)}    \int\limits_{N} \sum\limits_{j\in\mathbb{Z} } \sum\limits_{l\in\mathbb{Z}_+ }   F(\sigma^{j})   \varPhi(\sigma^{l+j})  q^{-|m \sigma^{l}\cdot o|/p}q^{-l-j}\,dm,$$
which after using Fubini's theorem  yields
\begin{equation}\label{eqn_T_Psi_bef_Abel}
	\left|	\left\langle T_{\Psi}^{+}f, \varphi\right\rangle\right|\leq C_{p} \| \Psi\|_{L^\infty(\mathfrak{X}\times S_p)}     \sum\limits_{j\in\mathbb{Z} } \sum\limits_{l \in \Z_+ }   F(\sigma^{j})   \varPhi(\sigma^{l+j})  q^{-l-j} \int\limits_{N}q^{-|m \sigma^{l}\cdot o|/p} \, dm.
\end{equation}
Since the Abel transform of a radial function is even (see \cite[Theorem 2.5]{Cowling_Meda_98}), it follows that  
$$\int\limits_{N}q^{-|m \sigma^{l}\cdot o|/p} \, dm= q^{l} \int\limits_{N}q^{-|m \sigma^{-l}\cdot o|/p} \, dm, \quad \text{ for all } l\in \Z_+.$$
Putting  the formula above in \eqref{eqn_T_Psi_bef_Abel}, we obtain
$$\left|	\left\langle T_{\Psi}^{+}f, \varphi\right\rangle\right| \leq C_{p} \| \Psi\|_{L^\infty(\mathfrak{X}\times S_p)}    \sum\limits_{j\in\mathbb{Z} } \sum\limits_{l \in \Z_+ }   F(\sigma^{j})   \varPhi(\sigma^{l+j})   q^{-j} \int\limits_{N}q^{-|m \sigma^{-l}\cdot o|/p} \, dm.$$
Using Lemma \ref{nproperties} (2) in the inequality above, it follows that
\begin{align*}
		\left|	\left\langle T_{\Psi}^{+}f, \varphi\right\rangle\right|  	& \leq C_{p} \| \Psi\|_{L^\infty(\mathfrak{X}\times S_p)}     \sum\limits_{j\in\mathbb{Z} } \sum\limits_{l \in \Z_+}   F(\sigma^{j})   \varPhi(\sigma^{l+j})   q^{-j-l/p} \int\limits_{N}q^{-|m \cdot o|/p} \,dm.
\end{align*}
Finally using Lemma \ref{qpn} and applying H\"older's inequality, we conclude that
\begin{align*}
\left|	\left\langle T_{\Psi}^{+}f, \varphi\right\rangle\right|  	& \leq C_{p} \| \Psi\|_{L^\infty(\mathfrak{X}\times S_p)}     \sum\limits_{j\in\mathbb{Z} } \sum\limits_{l \in \Z_+}   F(\sigma^{j})q^{-j/p}   \varPhi(\sigma^{j+l})   q^{-(j+l)/p^\prime}  q^{l(1/p^\prime-1/p)}\\
&\leq C_{p} \| \Psi\|_{L^\infty(\mathfrak{X}\times S_p)}   \|f\|_{L^p(NA)} \|\varphi\|_{L^{p^{\prime}}(NA)}.
\end{align*}
This completes the proof.
\end{proof}

\section{Boundedness of pseudo-differential operators, \texorpdfstring{$2<p<\infty$}{2pinfinity}}\label{bddness_of_T_Psi_p>2}
In order to prove Theorem \ref{mainresult2}, we will broadly follow the approach of the proof of Theorem \ref{thm:tpsi_plus_bdd}. To avoid the repetition of arguments, we shall only highlight the crucial steps.

\noindent\textbf{Proof of Theorem \ref{mainresult2}.} Our strategy is to prove that for any compactly supported functions $ f, \varphi :  NA \rightarrow \C $,
\begin{align}\label{eqn:to_prove_T_psi_bdd_p>2}
	\left| \left\langle T_{\Psi} f, \varphi\right\rangle\right|=  \left| \int\limits_{N}\sum\limits_{l\in\mathbb{Z} }   T_{\Psi}^{+}f(m\sigma^{l}) \varphi(m\sigma^l) q^{-l} dm\right|\leq C_p\|f\|_{L^p(NA)}  \|\varphi\|_{L^{p'}(NA)}.
\end{align}
As in  the case $ 1<p<2 $, we decompose  $T_\Psi f$ into two parts, namely $T_\Psi ^+ f$ and $ T_\Psi^- f $, where 
	$$ T_{\Psi}^{\pm}f(m\sigma^{l})= C  \int\limits_{N}\sum\limits_{j\in\mathbb{Z}}f(n\sigma^{j}) 
	\left(\int\limits_{\mathbb{T}}\Psi(m\sigma^l,s) \mathbf{c}(-s)^{-1}\,q^{(is-1/2)|\sigma^{-j}n^{-1}m\sigma^{l}\cdot o|}ds\right)\chi_{\pm}(l-j)q^{-j} \,dn.$$
We first  prove that $  T_{\Psi}^{+}$ satisfies  \eqref{eqn:to_prove_T_psi_bdd_p>2}. Sending the contour of integration in the inner integral from $ \T $ to $ \T+i\delta_{p} $ and noting that $\delta_{p}=\delta_{p^\prime} =1/p'-1/2$ (as $ p>2$), we get \eqref{eqn:equi_expn_of_T_psi_+} with $p$ being replaced by $ p^\prime $.
 Then  applying the same change of variables as in the proof of Theorem \ref{thm:tpsi_plus_bdd}, we obtain
\begin{multline}\label{eqn_equiv_defn_of_T_Psi_2}
	\left\langle T_{\Psi}^{+}f, \varphi\right\rangle = C \int\limits_{N}  \int\limits_{N} \sum\limits_{j\in\mathbb{Z} } \sum\limits_{l\in\mathbb{Z}_+ }   f(n\sigma^{j})  \varphi(n\sigma^j m \sigma^{l})q^{-|m \sigma^{l}\cdot o|/p^{\prime}}q^{-j-l}\\
	\cdot \left(\int\limits_{\mathbb{T}}\Psi_{\delta_{p}}(n\sigma^j m \sigma^{l},s)\mathbf{c}_{-\delta_{p}}(-s)^{-1} q^{is| m\sigma^{l}\cdot o|}~ds\right) \,dm\,dn.
\end{multline}
Observe that $h_{\omega^{+}_{0}}(m\sigma^{l}\cdot o)$ and hence $|m\sigma^l\cdot o|$ is strictly positive for all $l>0$. Consequently, integrating by parts the inner integrals of \eqref{eqn_equiv_defn_of_T_Psi_2} for all $l\in\mathbb{N}$, we deduce that
\begin{multline*}
	 \left\langle T_{\Psi}^{+}f, \varphi\right\rangle = C \int\limits_{N}  \int\limits_{N} \sum\limits_{j\in\mathbb{Z} } \sum\limits_{l \in \N }   f(n\sigma^{j})  \varphi(n\sigma^jm \sigma^{l})q^{-|m \sigma^{l}\cdot o|/p^{\prime}}q^{-j-l}\\
\cdot\frac{1}{|m\sigma^l\cdot o|^2} \left( \int\limits_{\mathbb{T}} \frac{d^2}{ds^2}\left(\Psi_{\delta_{p}}(n\sigma^j m \sigma^{l},s)\mathbf{c}_{-\delta_{p}}(-s)^{-1}\right) q^{is| m \sigma^{l}\cdot o|}~ds\right) \,dm\,dn\\
+ C \int\limits_{N}  \int\limits_{N} \sum\limits_{j\in\mathbb{Z} }   f(n\sigma^{j})  \varphi(n\sigma^jm ) q^{-j-|m \cdot o|/p^{\prime}}\\
\cdot  \left( \int\limits_{\mathbb{T}}  \Psi_{\delta_{p}}(n\sigma^j m,s)\mathbf{c}_{-\delta_{p}}(-s)^{-1} q^{is| m\cdot o|}~ds\right) \,dm\,dn.
\end{multline*}
Taking modulus on both the sides, using hypothesis \eqref{eqn_regular_condn_Psi_thm2}, the smoothness of $\mathbf{c}_{-\delta_{p}}(-s)^{-1}$ on $\T$, and then applying H\"older's inequality,  we get
\begin{multline}\label{eqn_2nd_abel} 
	\left|	\left\langle T_{\Psi}^{+}f, \varphi\right\rangle\right| \leq C_{p}  \left(~\sup\limits_{(x,z)\in\mathfrak{X}\times S_{p},~k=0,1,2}~\left|\frac{d^{k}}{dz^{k}}\Psi(x,z)\right|~\right)   \\  \cdot \left[\sum\limits_{j\in\mathbb{Z} } \sum\limits_{l \in \N }   F(\sigma^{j})   \varPhi(\sigma^{l+j})  q^{-j-l}
	 \int\limits_{N} \frac{q^{-|m \sigma^{l}\cdot o|/p^{\prime}}}{|m\sigma^l\cdot o |^2} dm \right.\\
	 \left. + \sum\limits_{j\in\mathbb{Z} }  F(\sigma^{j})   \varPhi(\sigma^{j})  q^{-j}  \int\limits_{N} q^{-|m\cdot o|/p^{\prime}} dm\right].
\end{multline}
Now using the evenness of  Abel transform of a radial function, it follows that  
$$\int\limits_{N}\frac{q^{-|m \sigma^{l}\cdot o|/p^{\prime}}}{|m\sigma^l\cdot o |^2} \, dm= q^{l} \int\limits_{N}\frac{q^{-|m \sigma^{-l}\cdot o|/p^{\prime}}}{|m\sigma^{-l}\cdot o |^2} \, dm, \quad \text{ for all } l\in \N.$$
Implementing the above fact in \eqref{eqn_2nd_abel} and following the same line of arguments as in the proof of Theorem \ref{thm:tpsi_plus_bdd}, we finally obtain
$$\left|	\left\langle T_{\Psi}^{+}f, \varphi\right\rangle\right| \leq C_{p} \left(~\sup\limits_{(x,z)\in\mathfrak{X}\times S_{p},~k=0,1,2}~\left|\frac{d^{k}}{dz^{k}}\Psi(x,z)\right|~\right)   \|f\|_{L^p(NA)} \|\varphi\|_{L^{p^{\prime}}(NA)}.$$
Next, we shall estimate $T_\Psi^-$. Proceeding in a similar way as in the proof of $T_{\Psi}^{+}$, we get
\begin{multline*}
	\left\langle T_{\Psi}^{-}f, \varphi\right\rangle = C \int\limits_{N}  \int\limits_{N} \sum\limits_{j\in\mathbb{Z} } \sum\limits_{l<0}   f(n\sigma^{j})  \varphi(n\sigma^j m \sigma^{l})q^{-|m \sigma^{l}\cdot o|/p^{\prime}}q^{-j-l}\\
	\cdot \left(\int\limits_{\mathbb{T}}\Psi_{\delta_{p}}(n\sigma^j m \sigma^{l},s)\mathbf{c}_{-\delta_{p}}(-s)^{-1} q^{is| m\sigma^{l}\cdot o|}~ds\right) \,dm\,dn.
\end{multline*}
Since $l<0$ in the equation above, using Lemma \eqref{nproperties} (2), we have
 \begin{multline*}
	\left\langle T_{\Psi}^{-}f, \varphi\right\rangle = C \int\limits_{N}  \int\limits_{N} \sum\limits_{j\in\mathbb{Z} } \sum\limits_{l<0}   f(n\sigma^{j})  \varphi(n\sigma^j m \sigma^{l})q^{(l-|m\cdot o|)/p^{\prime}}q^{-j-l}\\
	\cdot \left(\int\limits_{\mathbb{T}}\Psi_{\delta_{p}}(n\sigma^j m \sigma^{l},s)\mathbf{c}_{-\delta_{p}}(-s)^{-1} q^{is| m\sigma^{l}\cdot o|}~ds\right) \,dm\,dn.
\end{multline*}
An application of H\"older's inequality and \eqref{fphi} finally gives us
\begin{align*}
|\left\langle T_{\Psi}^{-}f, \varphi\right\rangle| &\leq C_{p}\|\Psi\|_{L^{\infty}(\mathfrak{X}\times S_{p})} \int\limits_{N} \sum\limits_{j\in\mathbb{Z} } \sum\limits_{l<0}   F(\sigma^{j})  \varPhi(\sigma^{l+j})q^{(l-|m\cdot o|)/p^{\prime}}q^{-j-l}\,dm\\
&\leq C_{p}\|\Psi\|_{L^{\infty}(\mathfrak{X}\times S_{p})} \|f\|_{L^p(NA)} \|\varphi\|_{L^{p^{\prime}}(NA)},
\end{align*}
which is the desired conclusion. \qed


\textbf{Acknowledgements.}  The authors would like to thank Prof. Pratyoosh Kumar and Prof. Sanjoy Pusti for many useful discussions during the course of this work. The first author was supported by research fellowship from CSIR (India). The second author gratefully acknowledges the support provided by NBHM (National Board of Higher Mathematics) post-doctoral fellowship (Number: 0204/3/2021/R$\&$D-II/7386) from the Department of Atomic Energy (DAE), Government of India. 
	

\bibliographystyle{plain}
\bibliography{References}

\end{document}